\documentclass[11pt]{article}

\usepackage{mathrsfs}
\usepackage{enumerate}
\usepackage[section]{algorithm}
\usepackage{algorithmic}
\usepackage{multirow}
\usepackage{xcolor}

\usepackage{graphicx}
\usepackage{amsmath,amsfonts,amsthm,amsbsy,amssymb}
\usepackage{fixmath}
\usepackage{amssymb,latexsym}

\usepackage{multirow}

\usepackage{hyperref}

\usepackage{cleveref}

\def\bp{\pmb{p}}

\def\bu{\pmb{u}}

\def\bbC{\mathbb{C}}

\def\bbR{\mathbb{R}}

\def\scrG{\mathscr{G}}

\def\cR{{\cal R}}
\def\cX{{\cal X}}
\def\cY{{\cal Y}}

\def\wtd{\widetilde}
\def\what{\widehat}

\usepackage{accents}

\newcommand\STM[2]{{\rm St}(#1,#2)}

\DeclareMathOperator{\diag}{diag}
\DeclareMathOperator{\dist}{dist}
\DeclareMathOperator{\eig}{eig}

\DeclareMathOperator{\rank}{rank}

\DeclareMathOperator{\tr}{tr}

\DeclareMathOperator{\F}{F}
\DeclareMathOperator{\HH}{H}
\DeclareMathOperator{\T}{T}

\newtheorem{theorem}{Theorem}[section]
\newtheorem{lemma}{Lemma}[section]
\newtheorem{corollary}{Corollary}[section]

\theoremstyle{definition}

\allowdisplaybreaks

\numberwithin{equation}{section}

\title{Approximations of Extremal Eigenspace and Orthonormal Polar Factor}

\author{Ren-Cang Li%
\thanks{Department of Mathematics, University of Texas at Arlington, Arlington, TX 76019-0408, USA.
        Email: {\tt rcli@uta.edu}.}
}

\date{October 10, 2025
}

\begin{document}

\maketitle

\begin{abstract}

This paper is concerned with two extremal  problems from matrix analysis. One is about approximating the top eigenspaces of a Hermitian matrix and the other one about approximating the orthonormal polar factor of a general matrix.
Tight error bounds on the quality of the approximations are obtained.


\end{abstract}

%
\section{Introduction}\label{sec:intro}
In this paper, we are concerned with two matrix extremal problems related to approximate the top eigenspaces of a Hermitian
matrix and the orthonormal polar factor of a matrix. Both problems will be stated in terms of certain matrix trace approximations.
The first problem  is well-known in matrix analysis while the second one isn't that much and only came up recently in \cite{li:2023arXiv}.
They played key roles  on optimization on Stiefel manifolds \cite{li:2023arXiv}, where the results
on the problems are used for the convergence analysis of two newly established frameworks for  numerical solutions
of most common optimization problems on Stiefel manifolds arising from various data science tasks.

Specifically, given a Hermitian matrix $H\in\bbC^{n\times n}$ and an orthonormal matrix $P\in\bbC^{n\times k}$,
our first problem is to bound the error in the column subspace of $P$ as an approximate eigenspace associated with
the $k$ largest eigenvalues of $H$, in terms of the difference between trace $\tr(P^{\HH}HP)$ and the sum of the $k$ largest eigenvalues.
There have been some results on this problem in the literature already \cite[Theorem~1]{kove:1991},
\cite[Chapter~3]{wein:1974}, \cite[Theorem~2.2]{li:2004c}. Our second problem is about the error analysis of
approximating the orthonormal polar factor of a matrix $B\in\bbC^{n\times k}$ via approximating its
trace norm, also known as the nuclear norm, by $\tr(P^{\HH}B)$ where $P\in\bbC^{n\times k}$ is orthonormal.
This second problem seems to appear for the first time in \cite{li:2023arXiv}.
Both problems are of interest in their own right.

The rest of this paper is organized as follows. In \cref{sec:angle-space}, we introduce the canonical angles between
two subspaces of equal dimension. \Cref{sec:eigenspace} is about approximating the extremal eigenspaces of
a Hermitian matrix while \cref{sec:singularspace} is about approximating the orthonormal polar factor of a tall matrix
by the means of approximating its nuclear norm. Finally, in \cref{sec:conclu} we draw our conclusions.

{\bf Notation.}
We follow conventional notation in matrix analysis.
The set $\bbR^{m\times n}$  denotes  $m\times n$ real matrices,
and $\bbR^n=\bbR^{n\times 1}$ and $\bbR=\bbR^1$.
Similarly, $\bbC^{m\times n}$,  $\bbC^n$, and $\bbC$ denote the
corresponding sets in the complex number field.
The identity matrix of size $n$ is denoted by $I_n$,
or simply $I$ if its size is clear from the context.
For any matrix or vector $B$,
$B^{\T}$ and $B^{\HH}$ stand for the transpose and  complex conjugate
transpose of $B$, respectively, and $\cR(B)$ denotes the column subspace, spanned by the columns of $B$.
The (complex) Stiefel manifold $\STM{k}{n}$ consists
of all $P\in\bbC^{n\times k}$ with orthonormal columns:
\begin{equation}\label{eq:stm} 
\STM{k}{n}=\{P\in\bbC^{n\times
k}\,:\,P^{\HH}P=I_k\}\subset\bbC^{n\times k}.
\end{equation}
A square matrix $A\succ 0$ (resp., $\succeq 0$) means that it is Hermitian and
positive definite (resp., semidefinite); accordingly $A\prec 0\, (\preceq
0)$ if $-A\succ 0\, (\succeq 0)$.
%
%
We use $\|B\|_2$ to denote the matrix 2-norm (i.e., the largest singular value of $B$)
and $\|B\|_{\F}:=\sqrt{\tr(B^{\HH} B)}$ the Frobenius norm.
Other notation will be explained at their first appearances.

\section{Canonical angles between subspaces}\label{sec:angle-space}
To serve the rest of this paper, we will introduce the canonical angles between two subspaces of equal dimension.

Let
$\cX=\cR(X)$ and $\cY=\cR(Y)$ be two $k$-dimensional subspaces of $\bbC^n$,
where $X,\,Y\in\STM{k}{n}$. The canonical angles
$\theta_1(\cX,\cY)\ge\cdots\ge\theta_k(\cX,\cY)$ between $\cX$ and $\cY$ are defined as \cite{stsu:1990}
$$
	0\le\theta_i\equiv\theta_i(\cX,\cY):=\arccos \sigma_i(X^{\HH}Y)\le\frac {\pi}2 \quad\mbox{for $1\le i\le k$},
$$
and accordingly,
the diagonal matrix of the canonical angles between $\cX$ and $\cY$ is
	$$
	\Theta(\cX,\cY)=\diag(\theta_1,\dots,\theta_k)\in\bbC^{k\times k}.
	$$
It is known that
	\begin{align}
	\dist_2(\cX,\cY)&:=\|\sin\Theta(\cX,\cY)\|_2=\sin\theta_1, \label{eq:sinTheta} \\
    \dist_{\F}(\cX,\cY)&:=\|\sin\Theta(\cX,\cY)\|_{\F}=\Big[\sum_{i=1}^k\sin^2\theta_i\Big]^{1/2} \label{eq:sinTheta:F}
	\end{align}
are two unitarily invariant metrics on the
Grassmann manifold
$\scrG_k(\bbC^n)$  \cite[p.99]{sun:2001}, the collection of all $k$-dimensional subspaces of $\bbC^n$ equipped with
the metric of either $\dist_2(\cdot,\cdot)$ or $\dist_{\F}(\cdot,\cdot)$.

\section{Extremal Eigenspace}\label{sec:eigenspace}
In this section, we will investigate the approximation of the top eigensapces of a Hermitian matrix.
We will establish one new result, i.e., the first inequality in \eqref{eq:eig2max} below.
For that purpose,
we will denote by
$\eig(H)=\{\lambda_i(H)\}_{i=1}^n$  the set of the
        eigenvalues (counted by multiplicities) of Hermitian matrix $H\in\bbC^{n\times n}$,
        arranged in the decreasing order:
        $$
        \lambda_1(H)\ge\lambda_2(H)\ge\cdots\ge\lambda_n(H).
        $$
The real version of \Cref{lm:maxtrace-proj} below is \cite[Lemma B.10]{li:2023arXiv}, and the only difference between them is the replacements of all matrix transposes in \cite[Lemma B.10]{li:2023arXiv} by the matrix complex conjugate transposes.

\begin{theorem}\label{lm:maxtrace-proj}
Let $H\in\bbC^{n\times n}$ be Hermitian and $P_*\in\STM{k}{n}$ whose column space $\cR(P_*)$ is the invariant subspace
of $H$ associated with its $k$ largest eigenvalues. Suppose that $\lambda_k(H)-\lambda_{k+1}(H)>0$.
Given $P\in\STM{k}{n}$, let
$$
\eta=\tr(P_*^{\HH}HP_*)-\tr(P^{\HH}HP), \quad \epsilon=\sqrt{\frac {\eta}{\lambda_k(H)-\lambda_{k+1}(H)}}.
$$
Then\footnote {The first inequality in \eqref{eq:eig2max} actually holds so long as
      $\cR(P_*)$ is a $k$-dimensional invariant subspace of $H$, as its proof will show. It is the second inequality
      that needs the condition of $\cR(P_*)$ being associated with the $k$ largest eigenvalues of $H$.}
\begin{equation}\label{eq:eig2max}
\frac {\|HP-P(P^{\HH}HP)\|_{\F}}{\lambda_1(H)-\lambda_n(H)}\le\|\sin\Theta(\cR(P),\cR(P_*))\|_{\F}\le {\epsilon}.
\end{equation}
\end{theorem}

\begin{proof}
We know that $\eta\ge 0$ by Fan's trace minimization principle \cite{fan:1949}.
The second inequality in \eqref{eq:eig2max} is \cite[Theorem~1]{kove:1991} and can also be derived from some of the estimates in
\cite[chapter~3]{wein:1974} and by a minor modification to the proof of \cite[Theorem~2.2]{li:2004c}. It remains to show
the first inequality in \eqref{eq:eig2max}.
Expand $P$ to $[P,P_{\bot}]\in\STM{n}{n}$. We find
$$
[P,P_{\bot}]^{\HH}\big(HP-P(P^{\HH}HP)\big)=\begin{bmatrix}
                                    0 \\
                                    P_{\bot}^{\HH}HP
                                  \end{bmatrix},
$$
implying
\begin{align}
\|HP-P(P^{\HH}HP)\|_{\F}&=\left\|[P,P_{\bot}]^{\HH}\big(HP-P(P^{\HH}HP)\big)\right\|_{\F} \nonumber \\
   &=\|P_{\bot}^{\HH}HP\|_{\F} \nonumber\\
   &=\|P_{\bot}^{\HH}(H-\xi I)P\|_{\F}, \label{eq:eig2max:pf-1}
\end{align}
for any $\xi\in\bbC$ because $P_{\bot}^{\HH}P=0$. But in what follows, we will take $\xi=[\lambda_1(H)+\lambda_n(H)]/2$.
Next we expand $P_*$ to $[P_*,P_{*\bot}]\in\STM{n}{n}$. Since the column space of $P_*$ is the invariant subspace
of $H$ associated with the $k$ largest eigenvalues of $H$, we have
$$
H[P_*,P_{*\bot}]=[P_*,P_{*\bot}]\begin{bmatrix}
                                  P_*^{\HH}HP_* &  \\
                                   & P_{*\bot}^{\HH}HP_{*\bot}
                                \end{bmatrix},
$$
and
\begin{align}
P_{\bot}^{\HH}(H-\xi I)P&=P_{\bot}^{\HH}[P_*,P_{*\bot}]\begin{bmatrix}
                                  P_*^{\HH}(H-\xi I)P_* &  \\
                                   & P_{*\bot}^{\HH}(H-\xi I)P_{*\bot}
                                \end{bmatrix}\begin{bmatrix}
                                               P_*^{\HH}\\ P_{*\bot}^{\HH}
                                             \end{bmatrix}P \nonumber\\
   &=P_{\bot}^{\HH}P_*P_*^{\HH}(H-\xi I)P_*P_*^{\HH}P
     +P_{\bot}^{\HH}P_{*\bot}P_{*\bot}^{\HH}(H-\xi I)P_{*\bot}P_{*\bot}^{\HH}P.  \label{eq:eig2max:pf-2}
\end{align}
Noticing that
\begin{gather*}
\|P_{\bot}^{\HH}P_*\|_{\F}=\|P_{*\bot}^{\HH}P\|_{\F}=\|\sin\Theta(\cR(P),\cR(P_*))\|_{\F}, \\
\|P_*^{\HH}P\|_2\le 1, \qquad\|P_{\bot}^{\HH}P_{*\bot}\|_2\le 1,
\end{gather*}
and, for $\xi=[\lambda_1(H)+\lambda_n(H)]/2$,
\begin{align*}
\|P_*^{\HH}(H-\xi I)P_*\|_2&\le\|H-\xi I\|_2=\frac 12 [\lambda_1(H)-\lambda_n(H)], \\
\|P_{*\bot}^{\HH}(H-\xi I)P_{*\bot}\|_2&\le\|H-\xi I\|_2=\frac 12 [\lambda_1(H)-\lambda_n(H)],
\end{align*}
we get from \eqref{eq:eig2max:pf-2},
\begin{align*}
\|P_{\bot}^{\HH}(H-\xi I)P\|_{\F}&\le 2\|H-\xi I\|_2\|\sin\Theta(\cR(P),\cR(P_*))\|_{\F} \\
    &=[\lambda_1(H)-\lambda_n(H)]\|\sin\Theta(\cR(P),\cR(P_*))\|_{\F},
\end{align*}
which together with \eqref{eq:eig2max:pf-1} yield the first inequality in \eqref{eq:eig2max}, as expected.
\end{proof}

%

\section{Orthonormal Polar Factor}\label{sec:singularspace}
In this section, we will bound the error in an approximation to
\begin{equation}\label{eq:approx-tr-norm}
\arg\max_{P\in\STM{k}{n}}\Re(\tr(P^{\HH}B)),
\end{equation}
which will be shown to be an orthonormal polar factor of $B$,
where $B\in\bbC^{n\times k}$ and $k\le n$, and $\Re(\cdot)$ extracts the real part of a complex number.
For notation, we will denote the singular values of $B$ by $\sigma_i(B)$ for $1\le i\le k$ and arranged in the descending order, i.e.,
$$
\sigma_1(B)\ge\sigma_1(B)\ge\cdots\sigma_k(B),
$$
the smallest singular value $\sigma_{\min}(B)=\sigma_k(B)$, and the trace norm $\|B\|_{\tr}=\sum_{i=1}^k\sigma_i(B)$, also known as the nuclear norm.
The following von Neumann's trace inequality is needed later in our proofs.

\begin{lemma}[von Neumann's trace inequality \cite{neum:1937}, {\cite[p.183]{hojo:1991}}]\label{lm:vN-tr-ineq}
For $B,\,C\in\bbC^{n\times k}$, we have
$$
|\tr(B^{\HH}C)|\le\sum_{i=1}^k\sigma_i(B)\,\sigma_i(C).
$$
\end{lemma}

The polar decomposition of $B\in\bbC^{n\times k}$
refers to $B=P\Lambda$ where $P\in\STM{k}{n}$ and $0\preceq\Lambda\in\bbC^{k\times k}$. It is known
that $\Lambda$ is unique and in fact $\Lambda=(B^{\HH}B)^{1/2}$ (the positive semidefinite square root of $B^{\HH}B$) but the orthonormal polar factor $P$ is unique if and only if $\rank(B)=k$
\cite{high:2008,li:1993b,li:1995,li:1997}, in which case  $P=B(B^{\HH}B)^{-1/2}$.
The polar decomposition can be  computed by the thin SVD: $B=U\Sigma V^{\HH}$, as
$P=UV^{\HH}$ and $\Lambda=V\Sigma V^{\HH}$,
where $U\in\STM{k}{n}$, $V\in\STM{k}{k}$, and
$\Sigma=\diag(\sigma_1(B),\ldots,\sigma_k(B))\in\bbR^{k\times k}$ is diagonal with the singular values  on the diagonal.

The next lemma are likely well-known. For example, its results are implied in the discussion in \cite{wazl:2022a} before
\cite[Lemma~3.2]{wazl:2022a} for the real number case (see also \cite[Lemma~B.11]{li:2023arXiv}, \cite{luli:2024}).
We provide a proof here for self-containedness and also because it is stated for the complex number field.

\begin{lemma}\label{lm:polar2max}
Let $B\in\bbC^{n\times k}$ and $k\le n$.
\begin{enumerate}[{\rm (a)}]
  \item $\Re(\tr(P^{\HH}B))\le\|B\|_{\tr}$ for any $P\in\STM{k}{n}$;
  \item $\Re(\tr(P^{\HH}B))=\|B\|_{\tr}$ where $P\in\STM{k}{n}$ if and only if $B=P\Lambda$ with $\Lambda\succeq 0$, a polar decomposition;
  \item We have
        $$
        \max_{P\in\STM{k}{n}}\Re(\tr(P^{\HH}B))=\|B\|_{\tr}
        $$
        and the optimal value $\|B\|_{\tr}$ is achieved by $P=P_*$, an orthonormal polar factor of $B$.
\end{enumerate}
\end{lemma}

\begin{proof}
Item (a) is a consequence of \Cref{lm:vN-tr-ineq}:
$$
\Re(\tr(P^{\HH}B))\le|\tr(P^{\HH}B)|\le\sum_{i=1}^k\sigma_i(P)\sigma_i(B)=\sum_{i=1}^k\sigma_i(B)=\|B\|_{\tr},
$$
where we have used $\sigma_i(P)=1$ for $1\le i\le k$.

For item (b), if $B=P\Lambda$ with $\Lambda\succeq 0$,
then $\tr(P^{\HH}B)=\tr(\Lambda)=\|B\|_{\tr}$. On the other hand, if $\Re(\tr(P^{\HH}B))=\|B\|_{\tr}$, then
immediately $0\le\tr(P^{\HH}B)\in\bbR$. Suppose $\rank(B)=r\le k$ and let the thin SVD of $B$ be
$$
B=U\Sigma V^{\HH}=U_1\Sigma_1V_1^{\HH},
$$
where $U_1\in\STM{r}{n}$, $V_1\in\STM{r}{k}$, and $\Sigma_1=\diag(\sigma_1(B),\ldots,\sigma_r(B))$.
Write
$$
\what P:=PV\equiv[\what \bp_1,\ldots,\what\bp_k]\in\STM{k}{n}, \quad
U\equiv [\bu_1,\ldots,\bu_k],
$$
where $\what \bp_i$ and $\bu_i$ are unit column vectors.
We have
$$
\Re(\tr(P^{\HH}B))=\Re(\tr(P^{\HH}U\Sigma V^{\HH}))
    =\Re(\tr([V^{\HH}P^{\HH}U]\Sigma ))
    =\sum_{i=1}^r\Re(\what\bp_i^{\HH}\bu_i)\sigma_i(B),
$$
yielding $\sum_{i=1}^r[1-\Re(\what\bp_i^{\HH}\bu_i)]\sigma_i(B)=0$ by $\Re(\tr(P^{\HH}B))=\|B\|_{\tr}=\sum_{i=1}^r\sigma_i(B)$. Hence
$$
\Re(\what\bp_i^{\HH}\bu_i)=1\quad\mbox{for $1\le i\le r$}
$$
because $1-\Re(\what\bp_i^{\HH}\bu_i)\ge 0$ and $\sigma_i(B)>0$ for  $1\le i\le r$. By the condition under which the Cauchy-Schwarz inequality
becomes an equality:
$$
1=\Re(\what\bp_i^{\HH}\bu_i)\le\|\what\bp_i\|_2\|\bu_i\|_2=1,
$$
we conclude that $\what\bp_i=\bu_i$ for  $1\le i\le r$. We can then write $\what P=[U_1,\what P_2]$ where
$\what P_2\in\bbC^{n\times (k-r)}$. Finally, we get
$$
P=\what PV^{\HH}=U_1V_1^{\HH}+\what P_2V_2^{\HH}\in\STM{k}{n}
$$
for which $P(V_1\Sigma_1 V_1^{\HH})=U_1\Sigma_1V_1^{\HH}=B$, a polar decomposition of $B$, as was to be shown.

Item (c) is a consequence of item (a) and item (b).
\end{proof}

Lemma~\ref{lm:polar2max} says that $\Re(\tr(P^{\HH}B))$  is bounded  above by $\|B\|_{\tr}$ always and the upper bound
$\|B\|_{\tr}$ is achieved
by any orthonormal polar factor $P_*$ of $B$ and also any maximizer of $\Re(\tr(P^{\HH}B))$ over $P\in\STM{k}{n}$ is an orthonormal polar factor of $B$.
Conceivably,
the closer $\Re(\tr(P^{\HH}B))$ is to the upper bound, the closer $P$ approaches to an orthonormal polar factor of $B$.
The results of the next lemma quantify the last statement.

\begin{theorem}\label{thm:polar2max}
Let $B\in\bbC^{n\times k}$ and suppose $\rank(B)=k$. Let $P_*$ be the unique
orthonormal polar factor of $B$.  Given $P\in\STM{k}{n}$, let
\begin{equation}\label{eq:eta-epsilon}
\eta=\|B\|_{\tr}-\Re(\tr(P^{\HH}B)), \quad
\epsilon=\sqrt{\frac {2\eta}{\sigma_{\min}(B)}}.
\end{equation}
\begin{enumerate}[{\rm (a)}]
  \item We have\footnote {The proof of the second inequality in \eqref{eq:polar2max} is actually through proving
            $\|\sin\frac 12\Theta(\cR(P),\cR(P_*))\|_{\F}\le \frac 12{\epsilon}$, which is stronger.
            Also since $\cR(P_*)$ is the same as $\cR(B)$ here, it can be replaced with $\cR(B)$.}
        \begin{equation}\label{eq:polar2max}
        \frac {\|B-P(P^{\HH}B)\|_{\F}}{\|B\|_2}\le\|\sin\Theta(\cR(P),\cR(P_*))\|_{\F}\le {\epsilon}\,;
        \end{equation}
  \item If  $P^{\HH}B\succ 0$, then
        \begin{equation} \label{eq:polar2max-2}
         \|P-P_*\|_{\F}\le\left(1+\frac {2\|B\|_2}{\sigma_{\min}(B)+\sigma_{\min}(P^{\HH}B)}\right)
                             {\epsilon}\,;
        \end{equation}
  \item  If  $\cR(P)=\cR(P_*)$, in which case $\sin\Theta(\cR(P),\cR(P_*))=0$, then
        \begin{equation} \label{eq:polar2max-3}
         \|P-P_*\|_{\F}\le {\epsilon}\,.
        \end{equation}
\end{enumerate}
\end{theorem}

\begin{proof}
Let $\theta_i$ for $1\le i\le k$ be the canonical angles between
subspaces $\cR(P)$ and $\cR(P_*)$ where $\pi/2\ge\theta_1\ge\cdots\ge\theta_k\ge 0$.
Then the singular values of $P^{\HH}P_*\in\bbC^{k\times k}$ are $\cos\theta_i$ for $1\le i\le k$. Let $B=P_*\Lambda$ be the polar decomposition of $B$.
We have
$P^{\HH}B=P^{\HH}P_*\Lambda$ and thus
\begin{equation}\label{eq:polar2max:pf-1}
\|B\|_{\tr}-\eta=\Re(\tr(P^{\HH}B))=\Re(\tr([P^{\HH}P_*]\Lambda))\le\sum_{i=1}^k\sigma_i(B)\cos\theta_{k-i+1}
\end{equation}
by Lemma~\ref{lm:vN-tr-ineq}.
Noticing $\|B\|_{\tr}=\sum_{i=1}^k\sigma_i(B)$, we get from \eqref{eq:polar2max:pf-1}
\begin{align}
\eta\ge\sum_{i=1}^k\sigma_i(B)\big[1-\cos\theta_{k-i+1}\big]&=\sum_{i=1}^k\sigma_i(B)\big[2\sin^2(\theta_{k-i+1}/2)\big]
       \label{eq:polar2max:pf-2}\\
  &\ge\sum_{i=1}^k\sigma_i(B)\cdot\frac 12\sin^2\theta_{k-i+1} \label{eq:polar2max:pf-3}\\
  &\ge\frac 12\sigma_{\min}(B)\|\sin\Theta(\cR(P),\cR(P_*))\|_{\F}^2, \nonumber
\end{align}
yielding the second inequality in \eqref{eq:polar2max}, where we have used
$$
\sin\theta\le 2\sin\frac {\theta}2=\frac {\sin\theta}{\cos(\theta/2)}\le\sqrt 2\sin\theta
\quad\mbox{for}\quad 0\le\theta\le\frac {\pi}2.
$$
Now we expand $P$ to $[P,P_{\bot}]\in\STM{n}{n}$. We find
$$
[P,P_{\bot}]^{\HH}\big(B-P(P^{\HH}B)\big)=\begin{bmatrix}
                                    0 \\
                                    P_{\bot}^{\HH}B
                                  \end{bmatrix},
$$
implying
\begin{equation}\label{eq:polar2max:pf-1'}
\|B-P(P^{\HH}B)\|_{\F}=\left\|[P,P_{\bot}]^{\HH}\big(B-P(P^{\HH}B)\big)\right\|_{\F}=\|P_{\bot}^{\HH}B\|_{\F}.
\end{equation}
It follows from the CS decomposition \cite{stsu:1990} that the singular values of $P_{\bot}^{\HH}P_*\in\bbC^{k\times (n-k)}$ comes from
$\sin\theta_i$ for $1\le i\le k$, with possibly some additional zeroes.
We have
$P_{\bot}^{\HH}B=P_{\bot}^{\HH}P_*\Lambda$ and hence
$$
\|P_{\bot}^{\HH}B\|_{\F}=\|P_{\bot}^{\HH}P_*\Lambda\|_{\F}
    \le\|P_{\bot}^{\HH}P_*\|_{\F}\|B\|_2=\|B\|_2\|\sin\Theta(\cR(P),\cR(P_*))\|_{\F},
$$
which, together \eqref{eq:polar2max:pf-1'}, yield the first inequality in \eqref{eq:polar2max}.

Next, we show \eqref{eq:polar2max-2}. The proof technique is borrowed from \cite[Lemma 4.1]{zhli:2014b} and
\cite[section 3.1]{teli:2025}.
Suppose now that $P^{\HH}B\succ 0$. Following the proof of \cite[Lemma 4.1]{zhli:2014b}, we
can conclude that there exists $Q\in\STM{k}{k}$ such that $\wtd P=P_*Q^{\HH}$ satisfies
\begin{equation}\label{eq:polar2max:pf-6}
\|P_*-\wtd P\|_{\F}^2=\sum_{i=1}^k 4\sin^2\frac {\theta_i}2
    \le\frac {2\eta}{\sigma_{\min}(B)},
\end{equation}
where we have used \eqref{eq:polar2max:pf-3} for the last inequality.
Adopting the argument in \cite[section~3.1]{teli:2025} upon noticing
$P^{\HH}B=I_k\cdot (P^{\HH}B)$ and $\wtd P^{\HH}B=Q\cdot (P_*^{\HH}B)$ are two polar decompositions, we have,
by \cite[Theorem~1]{li:1995},
\begin{align}
\|I_k-Q\|_{\F}&\le \frac 2{\sigma_{\min}(B)+\sigma_{\min}(P^{\HH}B)}\|P^{\HH}B-\wtd P^{\HH}B\|_{\F} \nonumber\\
   &\le \frac {2\|B\|_2}{\sigma_{\min}(B)+\sigma_{\min}(P^{\HH}B)}\|P-\wtd P\|_{\F}\,,  \label{eq:polar2max:pf-7}
\end{align}
and hence
$$
\|P-P_*\|_{\F}\le\|P-\wtd P\|_{\F}+\|\wtd P-P_*\|_{\F} \\
      =\|P-\wtd P\|_{\F}+\|I-Q\|_{\F},
$$
which, together with \eqref{eq:polar2max:pf-6} and \eqref{eq:polar2max:pf-7}, lead to \eqref{eq:polar2max-2}.
We may simply combine the second inequality in \eqref{eq:polar2max} with \cite[Theorem~3.1]{teli:2025}
to obtain a bound on $\|P-P_*\|_{\F}$, but then the resulting bound will be bigger than the right-hand side of \eqref{eq:polar2max-2} by
a factor of $\sqrt 2$.


Consider now item (c) for which $\cR(P)=\cR(P_*)$. Then $P=P_*W$ for some $W\in\STM{k}{k}$.
Recall that $B=P_*\Lambda$ is the polar decomposition of $B$, and hence  $\Lambda\succ 0$ and its eigenvalues are the singular values of $B$. Let $\Lambda=Q\Gamma Q^{\HH}$ be the eigendecomposition of $\Lambda$ where $Q\in\STM{k}{k}$ and
$\Gamma=\diag(\sigma_1(B),\ldots,\sigma_k(B))$.
Write $Q^{\HH}WQ=[w_{ij}]\in\STM{k}{k}$.
We know $|w_{ii}|\le 1$ for $1\le i\le k$.
We still have by \eqref{eq:polar2max:pf-1}
\begin{align*}
\eta&=\tr(\Lambda)-\Re(\tr(W^{\HH}\Lambda)) \\
    &=\tr(Q\Gamma Q^{\HH})-\Re(\tr(W^{\HH}Q\Gamma Q^{\HH})) \\
    &=\tr(\Gamma)-\Re(\tr(Q^{\HH}W^{\HH}Q\Gamma)),
\end{align*}
yielding
\begin{align}
&\eta=\sum_{i=1}^k[1-\Re(w_{ii})]\sigma_i(B)
    \ge\sigma_{\min}(B)\sum_{i=1}^k[1-\Re(w_{ii})], \nonumber \\
&\sum_{i=1}^k[1-\Re(w_{ii})]\le\frac {\eta}{\sigma_{\min}(B)}=\frac 12\epsilon^2. \label{eq:polar2max:pf-8}
\end{align}
We have
$\|P-P_*\|_{\F}^2=\|W-I\|_{\F}^2=\|Q^{\HH}(W-I)Q\|_{\F}^2=\|Q^{\HH}WQ-I\|_{\F}^2$, and thus
\begin{align*}
\|P-P_*\|_{\F}^2 
 &=\sum_{i=1}^k|w_{ii}-1|^2+\sum_{i=1}^k\sum_{j\ne i}|w_{ij}|^2 \\
 &=\sum_{i=1}^k|w_{ii}-1|^2+\sum_{i=1}^k(1-|w_{ii}|^2) \\
 &=2\sum_{i=1}^k[1-\Re(w_{ii})]
 \le\epsilon^2, \qquad(\mbox{by \eqref{eq:polar2max:pf-8}})
\end{align*}
as was to be shown.
\end{proof}


Both $P(P^{\HH}B)=(PP^{\HH})B$ and $\Theta(\cR(P),\cR(P_*))$ are invariant
with respect to substitution $PQ\leftarrow P$ for any $Q\in\STM{k}{k}$. Therefore $\epsilon$ in \eqref{eq:polar2max} can be replaced with the one so that $\Re(\tr([PQ]^{\HH}B))$ is maximized over $Q\in\STM{k}{k}$. Noticing that
\cite{luli:2024,wazl:2023}
$$
\max_{Q\in\STM{k}{k}}\Re(\tr([PQ]^{\HH}B))=\|P^{\HH}B\|_{\tr}
$$
and any orthonormal polar factor $Q$ of $P^{\HH}B$ is a maximizer for which $[PQ]^{\HH}B=Q^{\HH}(P^{\HH}B)\succeq 0$,
we arrive at the following corollary.

\begin{corollary}\label{cor:polar2max}
Let $B\in\bbC^{n\times k}$ and suppose $\rank(B)=k$. Let $P_*$ be the unique
orthonormal polar factor of $B$.
The inequalities in \eqref{eq:polar2max} hold with
\begin{equation}\tag{\ref{eq:eta-epsilon}$'$}
\eta=\|B\|_{\tr}-\|P^{\HH}B\|_{\tr}, \quad
\epsilon=\sqrt{\frac {2\eta}{\sigma_{\min}(B)}},
\end{equation}
and if also $\rank(P^{\HH}B)=k$ then
\begin{equation} \tag{\ref{eq:polar2max-2}$'$}
         \|PQ-P_*\|_{\F}\le\left(1+\frac {2\|B\|_2}{\sigma_{\min}(B)+\sigma_{\min}(P^{\HH}B)}\right)
                             {\epsilon}\,,
\end{equation}
where $Q$ is the orthonormal polar factor of $P^{\HH}B$.
\end{corollary}

\section{Concluding Remarks}\label{sec:conclu}

In this paper we investigated two matrix extremal problems concerning the eigenvalue problem of a Hermitian matrix and the
polar decomposition of a general matrix. Both problems are stated in terms of maximizing certain matrix trace functions.
The results have been used in \cite{li:2023arXiv} to analyze the convergence of two newly established frameworks
for numerical solutions of most common optimization problems on Stiefel manifolds. A notable difference is that
in \cite{li:2023arXiv} they are presented in the real number field whereas here it is for the complex number field.
Because of that, some subtleties arise. For example, we have to use $\Re(\cdot)$ to extract the real part of a complex number  at various places here.

\section*{Acknowledgment}
The research was supported in part by US NSF DMS-2407692.

{\small
\def\noopsort#1{}\def\l{\char32l}\def\v#1{{\accent20 #1}}
  \let\^^_=\v\def\hbk{hardback}\def\pbk{paperback}

}

\end{document}